\documentclass[10pt]{article}

\usepackage[small]{titlesec}
\usepackage[small,it]{caption}

\usepackage{amsmath}
\usepackage{amsthm}
\usepackage{amssymb}
\usepackage{latexsym}
\usepackage{pifont}
\usepackage{MnSymbol}
\usepackage{multirow}
\usepackage{appendix}

\usepackage{tikz}
\usetikzlibrary{arrows, automata, decorations.pathreplacing, fit, matrix, patterns, positioning}
\usepackage{color}
\definecolor{lightgray}{rgb}{0.6, 0.6, 0.6}
\definecolor{darkgray}{rgb}{0.7, 0.7, 0.7}
\definecolor{darkblue}{rgb}{0, 0, .4}

\usepackage[bookmarks]{hyperref}
\hypersetup{
        colorlinks=true,
        linkcolor=darkblue,
        anchorcolor=darkblue,
        citecolor=darkblue,
        urlcolor=darkblue,
        pdfpagemode=UseThumbs,
        pdftitle={Generating Permutations With Restricted Containers},
        pdfsubject={Combinatorics},
        pdfauthor={Albert, Homberger, Pantone, and Vatter},
}

\usepackage{todonotes}

\newcommand{\false}{\texttt{F}}
\newcommand{\true}{\texttt{T}}

\theoremstyle{plain}
\newtheorem{theorem}{Theorem}[section]
\newtheorem{proposition}[theorem]{Proposition}

\newtheorem{conjecture}[theorem]{Conjecture}

\theoremstyle{definition}

\makeatletter
\newfont{\footsc}{cmcsc10 at 8truept}
\newfont{\footbf}{cmbx10 at 8truept}
\newfont{\footrm}{cmr10 at 10truept}
\renewenvironment{abstract}%
                {
                  \begin{list}{}%
                     {\setlength{\rightmargin}{1in}%
                      \setlength{\leftmargin}{1in}}%
                   \item[]\ignorespaces\begin{small}}%
                 {\end{small}\unskip\end{list}}

\newcommand{\st}{\::\:}
\newcommand{\Av}{\operatorname{Av}}

\newcommand{\C}{\mathcal{C}}
\newcommand{\D}{\mathcal{D}}

\newcommand{\F}{\mathcal{F}}

\newcommand{\V}{\mathcal{V}}

\newcommand{\gr}{\mathrm{gr}}

\newcommand{\OEISlink}[1]{\href{http://oeis.org/#1}{#1}}
\newcommand{\OEISref}{\href{http://oeis.org/}{OEIS}~\cite{sloane:the-on-line-enc:}}
\newcommand{\OEIS}[1]{sequence \OEISlink{#1} in the \OEISref}
\newcommand{\Grid}{\operatorname{Grid}}

\newcommand{\p}[1]{#1^+}
\newcommand{\m}[1]{#1^-}
\renewcommand{\d}[1]{#1^{\bullet}}
\newpagestyle{main}[\small]{
        \headrule
        \sethead[\usepage][][]
        {\sc Generating Permutations With Restricted Containers}{}{\usepage}}

\title{\sc Generating Permutations With Restricted Containers}
\author{
	\begin{tabular}{cc}
        Michael Albert&Cheyne Homberger\\
		{\small Department of Computer Science}&{\small Department of Mathematics \& Statistics}\\[-3pt]
		{\small University of Otago}&{\small University of Maryland, Baltimore County}\\[-3pt]
		{\small Dunedin, New Zealand}&{\small Baltimore, Maryland}\\[20pt]
        Jay Pantone\footnotemark&Nathaniel Shar\\
		{\small Department of Mathematics}&{\small Department of Mathematics}\\[-3pt]
		{\small Dartmouth College}&{\small Rutgers University}\\[-3pt]
		{\small Hanover, New Hampshire}&{\small New Brunswick, New Jersey}\\[20pt]
		\multicolumn{2}{c}{Vincent Vatter\footnotemark[1]}\\
		\multicolumn{2}{c}{{\small Department of Mathematics}}\\[-3pt]
		\multicolumn{2}{c}{{\small University of Florida}}\\[-3pt]
		\multicolumn{2}{c}{\small Gainesville, Florida}
	\end{tabular}
}

\titleformat{\section}
        {\large\sc}
        {\thesection.}{1em}{}

\date{}

\begin{document}

\maketitle

\setcounter{footnote}{1}
\footnotetext{Pantone and Vatter were partially supported by the National Science Foundation under Grant Number DMS-1301692.}

\begin{abstract}
We investigate a generalization of stacks that we call $\C$-machines. We show how this viewpoint rapidly leads to functional equations for the classes of permutations that $\C$-machines generate, and how these systems of functional equations can be iterated and sometimes solved. 
General results about the rationality, algebraicity, and the existence of Wilfian formulas for some classes generated by $\C$-machines are given.
We also draw attention to some relatively small permutation classes which, although we can generate thousands of terms of their counting sequences, seem to not have D-finite generating functions.
\end{abstract}

\pagestyle{main}

\section{Introduction}

\subsection{History and context}

The study of permutation patterns is generally considered to have been started by Knuth, when he proved in the first volume of \emph{The Art of Computer Programming}~\cite[Section 2.2.1]{knuth:the-art-of-comp:1} that a permutation can be generated by a stack if and only if it avoids $312$ (i.e., does not contain three entries in the same relative order as $312$). That initial work was followed by a series of papers dealing with permutations sorted or generated by machines of varying kinds. Prompted apparently by an observation of Pratt~\cite{pratt:computing-permu:}, this led to the general consideration of \emph{permutation classes}, which has become a highly active area of combinatorial research in its own right. We refer to the last author's surery~\cite{vatter:permutation-cla:} for a broad overview of this area.

Many works in this area have focused on enumerative questions about permutation classes, as we do here. We would like to emphasize however that the ability to effectively enumerate a class is really a proxy indicating that we understand its structure at some sufficient level of detail --- one of the main goals of research in the area is to refine our understanding of such structure.

In this work we look back to the structure of permutations generated by machines, where a machine is simply a container subject to certain restrictions (analogous to the stack in Knuth's original work). We show that this notion captures the equivalence between various permutation classes, specifically most of those enumerated by the Catalan or Schr\"oder numbers. We are able to provide rigorous formulas that enumerate some other permutation classes of this type. Finally, because of the simplicity of the underlying mechanism we are able to compute a large number of initial terms in the generating functions of some other, seemingly equally uncomplicated, permutation classes and establish by empirical methods that their generating functions appear not to be D-finite.

\subsection{Concepts and definitions}

We are solely concerned with classical permutation patterns, in which the permutation $\pi$ \emph{contains} the permutation $\sigma$ if $\pi$ contains a subsequence \emph{order isomorphic} (i.e., with the same pairwise comparisons) to $\sigma$. Otherwise, $\pi$ \emph{avoids} $\sigma$. For example, $53412$ contains $321$, as evidenced by any of the subsequences $531$, $532$, $541$, or $542$. The containment relation is a partial order, and a \emph{permutation class} is a downset, or lower order ideal, of permutations under this order. 

As with any downset in a poset, every permutation class can be described as
\[
	\Av(B)=\{\pi\st\pi\mbox{ avoids all $\beta\in B$}\}
\]
for some set $B$ of permutations. We may take the set $B$ to be an \emph{antichain}, i.e., a set of pairwise incomparable permutations, and if $B$ is an antichain its choice is unique, and we refer to it as the \emph{basis} of the class in question. Given a permutation class $\C$ and nonnegative integer $n$, we denote by $\C_n$ the set of permutations in $\C$ of length $n$, and refer to
\[
	\sum_{n\ge 1} |\C_n|x^n=\sum_{\pi\in\C} x^{|\pi|}
\]
as its generating function. Two classes are \emph{Wilf-equivalent} if they have the same generating function. The \emph{growth rate} of the permutation class $\C$ is defined as
\[
	\gr(\C)=\lim_{n\rightarrow\infty} \sqrt[n]{|\C_n|}
\]
when this limit exists. 

Some permutation classes are trivially Wilf-equivalent via the \emph{symmetries} of the permutation order. Given a permutation $\pi = \pi(1)\pi(2)\cdots\pi(n)$, the reverse of $\pi$ is the permutation $\pi^{\textrm r}$ defined by $\pi^{\textrm r}(i) = \pi(n+1-i)$, the complement of $\pi$ is the permutation $\pi^{\textrm c}$ defined by $\pi^{\textrm c}(i) = n + 1 - \pi(i)$, and the (group-theoretic) inverse of $\pi$ is the permutation $\pi^{-1}$ defined by $\pi^{-1}(\pi(i)) = \pi(\pi^{-1}(i)) = i$. From the geometric viewpoint, reversing a permutation consists of reflecting its plot over any vertical line, complementing it consists of reflecting its plot over any horizontal line, and inverting it consists of reflecting its plot over a line of slope $1$.

We need to define the two operations on permutations illustrated in Figure~\ref{fig-sums}. Given permutations $\pi$ of length $k$ and $\sigma$ of length $\ell$, their \emph{(direct) sum} is defined as
\[
    (\pi\oplus\sigma)(i)
    =
    \left\{\begin{array}{ll}
    \pi(i)&\mbox{for $1\le i\le k$},\\
    \sigma(i-k)+k&\mbox{for $k+1\le i\le k+\ell$}.
    \end{array}\right.
\]
The analogous operation depicted on the right of Figure~\ref{fig-sums} is called the \emph{skew sum}. We can now characterize the class of permutations which can be generated by a $\C$-machine.

\begin{figure}
\begin{center}
	$\pi\oplus\sigma=$
	\begin{tikzpicture}[scale=0.5, baseline=(current bounding box.center)]
		\draw (0,0) rectangle (1,1);
		\draw (1,1) rectangle (2,2);
		\node at (0.5,0.5) {$\pi$};
		\node at (1.5,1.5) {$\sigma$};
	\end{tikzpicture}
\quad\quad\quad\quad
	$\pi\ominus\sigma=$
	\begin{tikzpicture}[scale=0.5, baseline=(current bounding box.center)]
		\draw (0,1) rectangle (1,2);
		\draw (1,0) rectangle (2,1);
		\node at (0.5,1.5) {$\pi$};
		\node at (1.5,0.5) {$\sigma$};
	\end{tikzpicture}
\end{center}
\caption{The sum and skew sum operations}
\label{fig-sums}
\end{figure}

We are concerned here with a fairly general family of machines. Suppose that $\C$ is a permutation class.
A $\C$-machine is a machine consisting of a container that holds partial permutations. In using this $\C$-machine to generate permutations from the input $12\cdots n$ we may at any time perform one of three operations:
\begin{itemize}
\item remove the next entry from the input and immediately append it to the end of the output (a \emph{bypass}),
\item remove the next entry from the input and place it anywhere in the container in such a way that the partial permutation in the container is in the same relative order as a permutation in the class $\C$ (a \emph{push}), or
\item remove the leftmost entry from the container and append it to the end of the output (a \emph{pop}).
\end{itemize}
The operation of this machine could be analogized to the situation of an administrator who, upon receiving a new task, may choose either to perform it immediately (the bypass option) or file it away. The administrator may also, of course, choose to perform some of the filed tasks, but only in the order in which they lie in the filing cabinet, and the possible orderings of the tasks within the filing cabinet is restricted.

We refer to a sequence of operations of this form as a \emph{generation sequence} for the permutation $\pi$ that is eventually produced. Formally, a generation sequence corresponds to a sequence of letters specifying which of these three actions was taken and in the case of a push operation, where the new element was pushed.

For a simple example, consider the $\Av(12)$-machine, illustrated on the left of Figure~\ref{fig-C-machine-example}. In this machine the container may only contain entries in decreasing order. Thus in generating permutations with the $\Av(12)$-machine, if we push an entry from the input to the container we must place it at the leftmost end of the container (because at any point in time all entries in the input are necessarily greater than every entry in the container). We may also pop from the beginning of the container. In this machine (but not in general) a bypass is equivalent to a push followed immediately by a pop, and therefore we may ignore bypasses. Thus the $\Av(12)$-machine is equivalent to a stack.

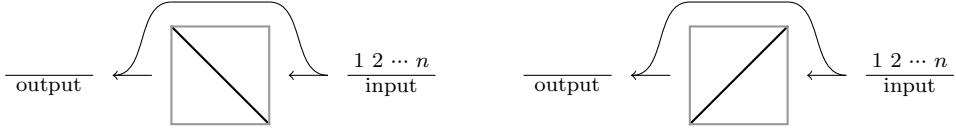
\begin{figure}
\begin{footnotesize}
\begin{center}
\begin{tabular}{ccc}
	\begin{tikzpicture}[scale=1.3, baseline=(current bounding box.center)]
		\draw [thick, line cap=round] (0,1)--(1,0);
		\foreach \i in {0,1}{
			\draw [lightgray, thick, line cap=round] (0,\i)--(1,\i);
		}
		\foreach \i in {0,1}{
			\draw [lightgray, thick, line cap=round] (\i,0)--(\i,1);
		}
		
		%
		\draw (1.8,0.5)--(2.7,0.5) node [above=2pt, below, midway] {input} node [above, midway] {$1\ 2\ \cdots\ n$};
		\draw (-1.7,0.5)--(-0.8,0.5) node [above=2pt, below, midway] {output};
		\draw[->] (1.6,0.5)--(1.2,0.5); 
		\draw (-0.2,0.5)--(-0.6,0.5);   
		\draw (1,1.25) to [out=0, in=180] (1.6, 0.5);
		\draw (1,1.25)--(0,1.25);
		\draw[->] (0,1.25) to [out=180, in=0] (-0.6, 0.5);
	\end{tikzpicture}
&\quad\quad&
	\begin{tikzpicture}[scale=1.3, baseline=(current bounding box.center)]
		\draw [thick, line cap=round] (0,0)--(1,1);
		\foreach \i in {0,1}{
			\draw [lightgray, thick, line cap=round] (0,\i)--(1,\i);
		}
		\foreach \i in {0,1}{
			\draw [lightgray, thick, line cap=round] (\i,0)--(\i,1);
		}
		
		%
		\draw (1.8,0.5)--(2.7,0.5) node [above=2pt, below, midway] {input} node [above, midway] {$1\ 2\ \cdots\ n$};
		\draw (-1.7,0.5)--(-0.8,0.5) node [above=2pt, below, midway] {output};
		\draw[->] (1.6,0.5)--(1.2,0.5); 
		\draw (-0.2,0.5)--(-0.6,0.5);   
		\draw (1,1.25) to [out=0, in=180] (1.6, 0.5);
		\draw (1,1.25)--(0,1.25);
		\draw[->] (0,1.25) to [out=180, in=0] (-0.6, 0.5);
	\end{tikzpicture}
\end{tabular}
\end{center}
\end{footnotesize}
\caption{The $\Av(12)$-machine, which generates $\Av(312)$, and the $\Av(21)$-machine, which generates $\Av(321)$. The internal lines in diagrams of this type represent the allowed positions of the elements stored in the machine, so in the first case any entries in the machine must be decreasing when read left to right, and in the second case they must be increasing.}
\label{fig-C-machine-example}
\end{figure}

Before beginning the general study of $\C$-machines, we consider one more specific example, the $\Av(21)$-machine, illustrated on the right of Figure~\ref{fig-C-machine-example}. In this machine we may only push into the rightmost end of the container, and since pops only occur from the far left of the machine, the bypass operation is necessary.  We claim that this machine generates the class $\Av(321)$. It is evidently impossible for the machine to generate $321$, as the $3$ would have to be the result of a bypass while the $21$ pattern lies in the container, which is not possible. In the other direction, we know that permutations in $\Av(321)$ consist of two increasing subsequences: their left-to-right maxima (the entries $\pi(j)$ satisfying $\pi(j)>\pi(i)$ for all $i<j$) and their non-left-to-right maxima. Upon reading the next symbol from the input, if it is to be a left-to-right maximum we first pop all entries in the container that come before it and then perform a bypass to put it in the correct position in the output, while if the next symbol from the input is not a left-to-right maximum we can simply push it into the container. When the input is empty, we finish by \emph{flushing} (popping all the entries of) the container.

It is well-known that $\Av(312)$ and $\Av(321)$ are both counted by the Catalan numbers, so the $\Av(21)$- and $\Av(12)$-machines generate equinumerous permutation classes. This is no accident. Indeed, Section~\ref{sec-catalan} shows how the description of $\Av(312)$ and $\Av(321)$ via $\C$-machines implicitly defines a bijection between these two classes which preserves the location and value of left-to-right maxima. This was observed in a similar context by Doyle~\cite{Doyle:Stackable-and-q:}.

\subsection{The main property}

It follows directly from the definitions that if $\C$ is a permutation class then so is the collection of permutations output from the $\C$-machine. The following theorem shows that there is a close connection between the bases of these two classes.

\begin{theorem}
\label{thm-basis}
For any set $B$ of permutations, the $\Av(B)$-machine generates the class
\[
	\Av(1\ominus B)=\Av(\{1\ominus\beta\st\beta\in B\}).
\]
\end{theorem}
\begin{proof}
The $\Av(B)$-machine cannot generate any permutation of the form $1\ominus\beta$ for $\beta\in B$; to do so, the container would have to contain a copy of $\beta$ at the point when the first entry of $1\ominus\beta$ was next in the input.

For the converse, suppose that $\pi$ avoids $1\ominus\beta$ for all $\beta\in B$. Label the positions of the left-to-right maxima of $\pi$ as $1=i_1<i_2<\cdots<i_k$. At the moment that $\pi(i_j)$ is the next symbol of the input, all entries which lie before it in $\pi$ are smaller than it (because it is a left-to-right maximum) so we may suppose that these entries have already exited or bypassed the container. Thus at this moment, the entries of $\pi$ which lie to the right and are smaller than $\pi(i_j)$ should be in the container. This is possible because these entries avoid all of the permutations in $B$ (because $\pi$ avoids $1\ominus B$). Thus upon reaching this point, we may bypass the container to place $\pi(i_j)$ directly in the output. We may then output all entries of the container which lie to the left of $\pi(i_{j+1})$ in $\pi$, and proceed as before. At the end of the process, we flush the container to complete the generation of $\pi$.
\end{proof}

The simple characterization provided by Theorem~\ref{thm-basis} allows us to tell immediately if a class is generated by a $\C$-machine: a class is generated by a $\C$-machine if and only if all of its basis elements begin with their largest entries. It also allows us to tell \emph{what} $\C$-machine generates a given class. In particular, let us consider a question raised by Mikl\'os B\'ona at the conference \emph{Permutation Patterns 2007}~\cite[Question 4]{vatter:problems-and-co:}. Atkinson, Murphy, and Ru\v{s}kuc~\cite{atkinson:sorting-with-tw:} showed that the permutation class sortable by two ``ordered'' stacks in series, despite having the infinite basis
\[
	\{2\ (2k-1)\ 4\ 1\ 6\ 3\ 8\ 5\ \cdots\ (2k)\ (2k-3) \st k\ge 2\},
\]
is equinumerous to the class $\Av(1342)$, first counted by B\'ona~\cite{bona:exact-enumerati:} (a simpler proof of this Wilf-equivalence result has since been given by Bloom and Vatter~\cite{bloom:two-vignettes-o:}). B\'ona asked
\begin{quote}
	``Is there a natural sorting machine / algorithm which can sort precisely the class $\Av(1342)$?''
\end{quote}
The answer (up to symmetry) is yes: the symmetric class $\Av(4213)$ is generated by the $\Av(213)$-machine. 

From the form of the basis for the class produced by a $\C$-machine (or from a simple consideration of its operation) all such classes are closed under the sum operation. It then follows immediately from Fekete's Lemma that all such classes have a proper growth rate (see Arratia~\cite{arratia:on-the-stanley-:}).

\subsection{Operation of the machine}
\label{sec-unique}

For any class $\C$, the $\C$-machine seemingly has three operations at its disposal: bypass, push, and pop. For enumerative purposes we must establish a \emph{unique generation sequence} for every permutation that can be generated. To this end, we adopt conventions that handle two situations where non-uniqueness could arise:
\begin{enumerate}
\item[(U1)] we should pop from the container whenever possible, and
\item[(U2)] we should bypass the container whenever possible.
\end{enumerate}
Another way to phrase these two rules is that in trying to generate a particular permutation $\pi$ from the input sequence $12 \cdots n$ if at some point the next symbol of $\pi$ is the first symbol in the container then it should be output immediately (by U1), while if it is the next symbol on the input then it should be produced by an immediate bypass (by U2).
The rules (U1) and (U2) correspond to choosing the ``leftmost'' possible action at all times. Another valuable observation is that (U1) and (U2) together imply that in any generation sequence, no pop will immediately follow a push, because otherwise the pop should have either been a bypass or occurred earlier. In our resulting functional equations, this issue will frequently arise as a flag which indicates whether pops are permitted in the corresponding state. Our next result verifies that (U1) and (U2) indeed guarantee uniqueness.

\begin{proposition}
\label{prop-uniqueness}
For any class $\C$ and any permutation $\pi$ that can be generated by the $\C$-machine, there is a unique generation sequence satisfying (U1) and (U2) that produces $\pi$ from the $\C$-machine. Moreover, the left-to-right maxima of $\pi$ are exactly those symbols produced by bypass operations.
\end{proposition}
\begin{proof}
Suppose that $\pi$ can be generated by the $\C$-machine. Clearly we can find a generation sequence for $\pi$ which satisfies (U1) and (U2), so it suffices to show that this generation sequence is uniquely determined by $\pi$, (U1), and (U2).

At the point when $\pi(i)$ is the next symbol in the input, all smaller symbols lie either in the container or the output. By (U1), we must first pop all symbols that we can before doing anything to $\pi(i)$. By (U2), if $\pi(i)$ is a left-to-right maximum, it must bypass the container. Otherwise $\pi(i)$ is not a left-to-right maximum so it must be pushed into the container (since it is preceded in $\pi$ by some greater symbol which is still in the input) and since container symbols are output in left-to-right order its placement relative to the other entries (if any) currently held in the container is uniquely determined by its position in $\pi$.
\end{proof}

\subsection{Structure of the paper}

As will be demonstrated, if a class can be generated by a $\C$-machine, then it is fairly automatic to use the machine to determine a set of functional equations (including catalytic variables) for its generating function. Roughly speaking the remainder of this paper consists of an exploration of that claim in a series of examples where dealing with the resulting generating functions becomes more and more difficult.

We begin with the classical cases enumerated by the Catalan and Schr\"oder numbers in Section~\ref{sec-catalan}. Here the solutions of the functional equations can be derived easily using the kernel method. There is still value in bringing the $\C$-machine context into play here as it establishes \emph{uniform} arguments for these Wilf-equivalences and also provides bijections between the corresponding classes that preserve both the values and positions of the left-to-right maxima.

In Section~\ref{sec-other-examples} we consider two cases (the ``Fibonacci machines'') where algebraic solutions of the functional equations can still be rigorously obtained.

Section~\ref{sec-finite-bounded} establishes some general results which show that for very small classes $\C$, the corresponding classes produced by their $\C$-machines have easily computed enumerations.

Already in some of the earlier sections, but particularly when we proceed to Section~\ref{sec-non-D-finite}, it is clear that while the translation from $\C$-machines to functional equations is `fairly' automatic, it can require some effort to simplify these functional equations into a form that either permits a solution or allows for the efficient generation of a large number of terms. In the latter case it is often more useful to consider how to represent the internal state of the $\C$-machine using a small number of parameters. When this is possible, dynamic programming approaches allow for the generation of thousands of initial terms of the corresponding sequences.

Alas, sometimes this simplification proves impossible, as for the notoriously unenumerated class $\Av(4231)$. While we may view this class as the output of the $\Av(231)$-machine, that perspective does not improve our knowledge of its enumeration. However, the $\C$-machine approach does allow us to compute a great number of terms for some of its subclasses. To pick an example we find particularly alluring, in Section~\ref{sec-non-D-finite} we show how to generate 5,000 terms of the enumeration of the class $\Av(4231,4123,4312)$. Despite the abundance of data we have for this example, we are not able to fit its generating function to any algebraic differential equation. Interestingly this means that in the chain of classes
\[
	\Av(4231,4123,4312)
	\subset
	\Av(4231,4312)
	\subset
	\Av(4231),
\]
the first class is easy to enumerate (we can compute terms in polynomial time) seems to lack a D-finite generating function, the second has an algebraic generating function (see Section~\ref{sec-schroeder} where we analyze it as a $\C$-machine), and the third seems very difficult to enumerate (the current record is 50 terms, computed by Conway, Guttmann, and Zinn-Justin~\cite{conway:1324-avoiding-p:} building on the work of Johansson and Nakamura~\cite{johansson:using-functiona:} and Conway and Guttmann~\cite{conway:on-the-growth-r:}).

Noonan and Zeilberger~\cite{noonan:the-enumeration:} conjectured in 1996 that every finitely based permutation class has a D-finite generating function. Zeilberger changed his mind about the conjecture less than a decade later (see \cite{elder:problems-and-co:}) and Garrabrant and Pak~\cite{Garrabrant:Pattern-avoidan:} have recently disproved it. We believe that the class
\[
	\Av(4231,4123,4312)
\]
represents a good candidate to be the first \emph{concrete} counterexample to the false conjecture.

\section{Catalan and Schr\"oder classes}
\label{sec-catalan}

In this section we show how the perspective of $\C$-machines allows us to define length-preserving bijections between a number of classes enumerated by the Catalan and Schr\"oder numbers. While this ground is well-trodden, we believe that at the very least the $\C$-machine perspective presents a particularly straightforward view of these Wilf-equivalences.


\subsection{Catalan classes}

The rules (U1) and (U2) of Section~\ref{sec-unique} implicitly give a bijection between $\Av(312)$ and $\Av(321)$. When generating a permutation with either the $\Av(12)$- or $\Av(21)$-machine, we must pop whenever possible and all left-to-right maxima must bypass the container. Moreover, since in either case the contents of the container must form a monotone sequence, whenever we push into the container, there is a unique position for the new entry to be placed. In fact, this argument establishes that there is a unique bijection between $\Av(312)$ and $\Av(321)$ that preserves the locations and values of left-to-right maxima. This bijection is equivalent, by symmetry, to one presented by Knuth~\cite{knuth:the-art-of-comp:1}.

Note that this bijection also restricts to a bijection between permutations that can be generated by the $\Av(12,k\cdots 21)$- and $\Av(21,12\cdots k)$-machines, implying that the classes $\Av(312,(k+1)\dots 21)$ and $\Av(321, (k+1)12\dots k)$ are Wilf-equivalent. This result was first established by Chow and West~\cite{chow:forbidden-subse:}, who showed that the generating functions of these classes are quotients of Chebyshev polynomials. Of course, these generating functions simply count Dyck paths of maximum height $k$. 

We now consider a different viewpoint which will become necessary when we analyze more complicated machines. We can think of the $\Av(21)$-machine operating under the rules (U1) and (U2) as being in one of two states that we call ``can pop'' and ``can't pop''. The machine is in the ``can't pop'' state whenever we have just pushed a symbol into the container and in the ``can pop'' state at all other times, as shown in Figure~\ref{fig-Av21-automata}.
%
%

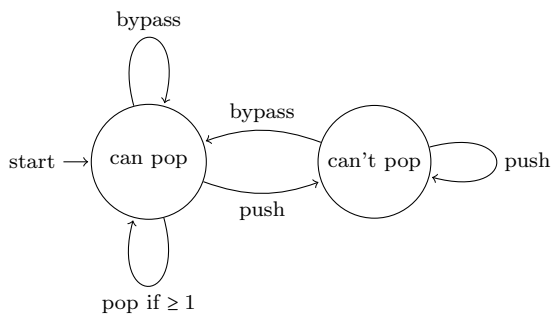
\begin{figure}
	\begin{footnotesize}
	\begin{center}
		\begin{tikzpicture}[shorten >=1pt,node distance=3cm,on grid,auto] 
			\node[state,initial] (canpop)   {\hspace{.42em}can pop\hspace*{.42em}};
			\node[state] (cantpop) [right =of canpop] {can't pop}; 
			\path[->] 
				(canpop) edge[out=-20,in=200] node[swap, sloped, anchor=north] {push} (cantpop)
					edge [loop above] node[anchor=south] {bypass} ()
					edge [loop below] node[anchor=north] {pop if $\geq1$} ()
				(cantpop) edge[out=160,in=20] node[sloped, anchor=south] {bypass} (canpop)
					edge [loop right] node[anchor=west] {push} ();
		\end{tikzpicture}
	\end{center}
	\end{footnotesize}
	\caption{An automaton representing the $\Av(21)$-machine}
	\label{fig-Av21-automata}
\end{figure}

Let $f(x,u)$ denote the generating function for paths to the ``can pop'' state, where $x$ tracks the number of output symbols, and $u$ tracks the number of symbols in the container. Also let $g(x,u)$ denote the generating function for paths to the ``can't pop'' state with the same variables. By considering all possible transitions among these two states, we derive the system of equations
\begin{eqnarray*}
	f(x,u)&=&1+x(f(x,u)+g(x,u))+\frac{x}{u}(f(x,u)-f(x,0)),\\
	g(x,u)&=&u(f(x,u)+g(x,u)).
\end{eqnarray*}

Then by a simple application of the kernel method\footnote{In fact, the original inspiration for the kernel method came from Knuth's solution~\cite[Solution 2.2.1.4]{knuth:the-art-of-comp:1} to this very problem (though he did not use this language or the same functional equation).} we get
\[
f(x,0)=(1-\sqrt{1-4x})/2x.
\]


%
%
%
%
%
%
%
%

\subsection{The Schr\"oder Classes}
\label{sec-schroeder}

It is an easy computation to show that the classes defined by avoiding two patterns of length four (the \emph{2$\times$4 classes}) form $56$ symmetry classes. After a significant amount of work~\cite{bona:the-permutation:,kremer:permutations-wi:,kremer:postscript:-per:,kremer:finite-transiti:,le:wilf-classes-of:}, it has been shown that these $56$ symmetry classes fall into $38$ Wilf equivalence classes, of which explicit generating functions have been found for all but $5$. By far the largest of these Wilf equivalence classes consists of $10$ symmetry classes enumerated by the Schr\"oder numbers (this Wilf equivalence class was found by Kremer~\cite{kremer:permutations-wi:,kremer:postscript:-per:}). Of these $10$ symmetry classes, $6$ can be generated by $\C$-machines, in a completely parallel manner, as we describe in this section.

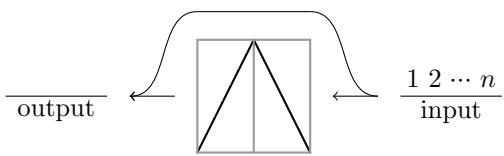
\begin{figure}
\begin{center}
	\begin{tikzpicture}[scale=1.5, baseline=(current bounding box.center)]
		\draw [thick, line cap=round] (0,0)--(0.5,1);
		\draw [thick, line cap=round] (0.5,1)--(1,0);
		\foreach \i in {0,1}{
			\draw [lightgray, thick, line cap=round] (0,\i)--(1,\i);
		}
		\foreach \i in {0,0.5,1}{
			\draw [lightgray, thick, line cap=round] (\i,0)--(\i,1);
		}
		
		%
		\draw (1.8,0.5)--(2.7,0.5) node [above=2pt, below, midway] {input} node [above, midway] {$1\ 2\ \cdots\ n$};
		\draw (-1.7,0.5)--(-0.8,0.5) node [above=2pt, below, midway] {output};
		\draw[->] (1.6,0.5)--(1.2,0.5); 
		\draw[->] (-0.2,0.5)--(-0.6,0.5);   
		\draw (1,1.25) to [out=0, in=180] (1.6, 0.5);
		\draw (1,1.25)--(0,1.25);
		\draw[->] (0,1.25) to [out=180, in=0] (-0.6, 0.5);
	\end{tikzpicture}	
\end{center}
\caption{The $\Av(312,213)$-machine generates a Schr\"oder class.}
\label{fig-Av-312-213-machine}	
\end{figure}

The first Schr\"oder class we consider is $\Av(4312,4213)$, which is generated by the $\Av(312,213)$-machine shown in Figure~\ref{fig-Av-312-213-machine}. As indicated by the dark lines in this figure, permutations in $\Av(312,213)$ consist of an increasing sequence followed by a decreasing sequence.

By (U1) and (U2), pops and bypasses in the $\Av(312,213)$-machine function the same as they do in the $\Av(21)$-machine, but pushes function differently. When the container is empty there is only one position to push into, and when the container is nonempty there are two positions to push into: either immediately to the left of the maximum entry in the container or immediately to the right of this entry. By making a small variation to the functional equations for the $\Av(21)$-machine, we are led to the system
\begin{eqnarray*}
	f(x,u)&=&1+x(f(x,u)+g(x,u))+\frac{x}{u}(f(x,u)-f(x,0)),\\
	g(x,u)&=&2u((f(x,u)-f(x,0))+g(x,u))+uf(x,0).
\end{eqnarray*}
Here the $f(x,u)$ equation has stayed the same, but the $g(x,u)$ equation has changed to reflect the number of positions we may push into.

This example is sufficiently simple to solve by hand using the kernel method and yields
\[
	f(x,0)=\frac{3-x-\sqrt{1-6x+x^2}}{2}. 
\]

%

\begin{figure}
\begin{footnotesize}
\begin{center}
\begin{tabular}{ccccc}
	\begin{tikzpicture}[scale=1.5, baseline=(current bounding box.center)]
		\draw [thick, line cap=round] (0,1)--(0.5,0);
		\draw [thick, line cap=round] (0.5,0)--(1,1);
		\foreach \i in {0,1}{
			\draw [lightgray, thick, line cap=round] (0,\i)--(1,\i);
		}
		\foreach \i in {0,0.5,1}{
			\draw [lightgray, thick, line cap=round] (\i,0)--(\i,1);
		}
		
	\end{tikzpicture}	
&\quad\quad&
	\begin{tikzpicture}[scale=1.5, baseline=(current bounding box.center)]
		\draw [thick, line cap=round] (0,0.25)--(0.25,0);
		\draw [thick, line cap=round] (0.25,0.5)--(0.5,0.25);
		\draw [thick, line cap=round] (0.5,0.75)--(0.75,0.5);
		\foreach \i in {0,0.25,0.5,0.75,1}{
			\draw [lightgray, thick, line cap=round] (0,\i)--(1,\i);
		}
		\foreach \i in {0,0.25,0.5,0.75,1}{
			\draw [lightgray, thick, line cap=round] (\i,0)--(\i,1);
		}
		\useasboundingbox (current bounding box.south west) rectangle (current bounding box.north east);
		
		\node [rotate=45] at (0.89,0.89) {{\footnotesize $\dots$}};
	\end{tikzpicture}
&\quad\quad&
	\begin{tikzpicture}[scale=1.5, baseline=(current bounding box.center)]
		\draw [thick, line cap=round] (0.25,0.5)--(0.5,0.75);
		\draw [thick, line cap=round] (0.5,0.25)--(0.75,0.5);
		\draw [thick, line cap=round] (0.75,0)--(1,0.25);
		\foreach \i in {0,0.25,0.5,0.75,1}{
			\draw [lightgray, thick, line cap=round] (0,\i)--(1,\i);
		}
		\foreach \i in {0,0.25,0.5,0.75,1}{
			\draw [lightgray, thick, line cap=round] (\i,0)--(\i,1);
		}
		\useasboundingbox (current bounding box.south west) rectangle (current bounding box.north east);
		\node [rotate=-45] at (0.13,0.87) {{\footnotesize $\dots$}};
		
	\end{tikzpicture}\\[23pt]
$\Av(132,231)$
&&
$\Av(312,231)$
&&
$\Av(213,132)$\\[15pt]	
\quad\quad&
	\begin{tikzpicture}[scale=1.5, baseline=(current bounding box.center)]
		\draw [thick, line cap=round] (0,0)--(0.75,0.75);
		
		\foreach \i in {0,0.25,0.5,0.75,1}{
			\draw [lightgray, thick, line cap=round] (0,\i)--(1,\i);
		}
		\foreach \i in {0,0.25,0.5,0.75,1}{
			\draw [lightgray, thick, line cap=round] (\i,0)--(\i,1);
		}
		\useasboundingbox (current bounding box.south west) rectangle (current bounding box.north east);
		
		\draw[fill=black] (0.25,0) circle (1pt);
		\draw[fill=black] (0.5,0.25) circle (1pt);
		\draw[fill=black] (0.75,0.5) circle (1pt);
		\node [rotate=45] at (0.89,0.89) {{\footnotesize $\dots$}};
	\end{tikzpicture}	
&\quad\quad&
	\begin{tikzpicture}[scale=1.5, baseline=(current bounding box.center)]
		\draw [thick, line cap=round] (0.25,0.75)--(1,0);
		\foreach \i in {0,0.25,0.5,0.75,1}{
			\draw [lightgray, thick, line cap=round] (0,\i)--(1,\i);
		}
		\foreach \i in {0,0.25,0.5,0.75,1}{
			\draw [lightgray, thick, line cap=round] (\i,0)--(\i,1);
		}
		\useasboundingbox (current bounding box.south west) rectangle (current bounding box.north east);
		\node [rotate=-45] at (0.13,0.87) {{\footnotesize $\dots$}};
		\draw[fill=black] (0.25,0.5) circle (1pt);
		\draw[fill=black] (0.5,0.25) circle (1pt);
		\draw[fill=black] (0.75,0) circle (1pt);
	\end{tikzpicture}
\\[23pt]
&
$\Av(321,312)$
&&
$\Av(213,123)$
\end{tabular}
\end{center}
\end{footnotesize}
\caption{Five classes whose machines generate Schr\"oder classes}
\label{fig-more-schroder}
\end{figure}
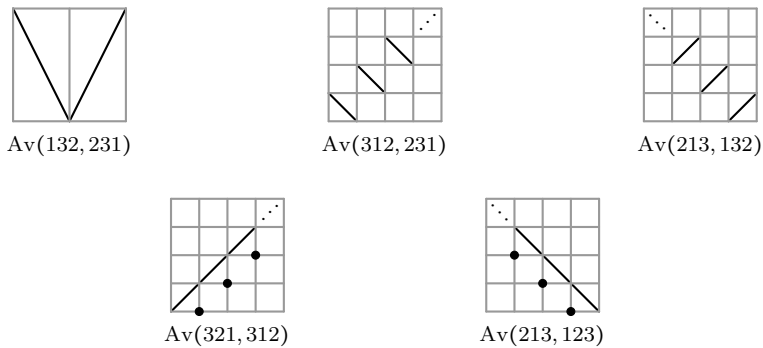


\begin{proposition}
The permutation classes $\Av(4312,$ $4213)$, $\Av(4132,$ $4231)$, $\Av(4312,$ $4231)$, $\Av(4213,$ $4132)$, $\Av(4321,$ $4312)$, and $\Av(4213,$ $4123)$ are all enumerated by the Schr\"oder numbers. Furthermore, there are length-preserving bijections between them that also preserve the location and value of the left-to-right maxima.
\end{proposition}

\begin{proof}
Figure~\ref{fig-Av-312-213-machine} and Figure~\ref{fig-more-schroder} show the six classes whose associated machines generate the Schr\"oder classes listed. In each case it is evident that, when the machine contains one or more elements there are exactly two ways to insert a new maximum element. Recall that Proposition~\ref{prop-uniqueness} guarantees that given any permutation $\pi$ that can be generated by a $\C$-machine, there is a unique generation sequence satisfying (U1) and (U2). Since all these machines have the same generation sequences this provides length preserving bijections between the corresponding classes produced, and these bijections preserve the location and values of left-to-right maxima because these are precisely the elements of the output produced by bypass operations.
\end{proof}

As far as we are aware, bijections between these classes preserving the location and values of the left-to-right maxima have not been presented in the literature before, although their existence also follows from the results of Bloom and Elizalde~\cite[Section 6]{bloom:pattern-avoidan:}.



The remaining four Schr\"oder classes are:
\[
\Av(4312, 3412), \Av(4213, 2413), \Av(4213, 3214), \Av(3142, 2413).
\]
None of these (or any of their symmetries) have both basis elements beginning with a $4$ and therefore they cannot be enumerated by the mechanisms of $\C$-machines as we are presenting them here. While it might be possible to generalize or vary the operation of $\C$-machines to account for these classes, it is easy to see by direct computation that there cannot be bijections that preserve the positions and values of left-to-right maxima between the six Schr\"oder classes we have considered and the other four Schr\"oder classes. Thus such variations would at the very least modify the type of permutation statistics that are preserved by a consideration of distinct $\C$-machines operating with a common operation sequence.


\section{Fibonacci Machines}
\label{sec-other-examples}

Here we consider the class of permutations $\F_\oplus$ formed by sums of the permutations $1$ and $21$ and the symmetric class of permutations $\F_\ominus$ formed by skew sums of the permutations $1$ and $12$ as shown in Figure~\ref{fig-fib-classes} (the presence of two dots in each cell indicates that we may put zero, one, or two entries in each cell). We call these classes the \emph{Fibonacci classes}, as the number of permutations of length $n$ in each class is the $n$th Fibonacci number, $F_n$, with initial conditions $F_0 = F_1 = 1$.

At first glance it might seem that if two permutation classes $\C$ and $\D$ are Wilf-equivalent then so should be the classes generated by the $\C$-machine and the $\D$-machine. However, the operation of a machine is not symmetric: input arrives at the top, but output is produced from the left. Unless there is a bijection underlying the original Wilf-equivalence that respects this asymmetry, the corresponding machines will behave differently. 
To give a concrete example of why this is the case with respect to the $\F_\oplus$- and $\F_\ominus$-machines, suppose we fill the $\F_\oplus$-machine with $2143$, then perform a bypass, and then a pop. The container will then hold $143$, and there is a unique way to perform a push, then a bypass, and then empty the machine. On the other hand, the analogous generation sequence applied to the $\F_\ominus$-machine would tell us to fill it with $3412$, then perform a bypass and a pop. At that point the container will hold $412$, leaving us with two ways to perform a push (resulting in either $5412$ or $4512$), then a bypass, and then to empty the machine.

\begin{figure}
\begin{center}
\begin{tabular}{ccc}
		\begin{tikzpicture}[scale=1.3, baseline=(current bounding box.center)]
		\foreach \i in {0,0.25,0.5,0.75,1}{
			\draw [lightgray, thick, line cap=round] (0,\i)--(1,\i);
		}
		\foreach \i in {0,0.25,0.5,0.75,1}{
			\draw [lightgray, thick, line cap=round] (\i,0)--(\i,1);
		}
		\foreach \i in {0,.25,.5}{
			\draw[fill=black] (0.08333+\i,0.16667+\i) circle (1pt);
			\draw[fill=black] (0.16667+\i,0.08333+\i) circle (1pt);
		}
		%
		\draw (1.8,0.5)--(2.7,0.5) node [above=2pt, below, midway] {input} node [above, midway] {$1\ 2\ \cdots\ n$};
		\draw (-1.7,0.5)--(-0.8,0.5) node [above=2pt, below, midway] {output};
		\draw[->] (1.6,0.5)--(1.2,0.5); 
		\draw[->] (-0.2,0.5)--(-0.6,0.5);   
		\draw (1,1.25) to [out=0, in=180] (1.6, 0.5);
		\draw (1,1.25)--(0,1.25);
		\draw[->] (0,1.25) to [out=180, in=0] (-0.6, 0.5);
		\node [rotate=45] at (0.89,0.89) {{\footnotesize $\dots$}};
	\end{tikzpicture}	
&\quad\quad&
	\begin{tikzpicture}[scale=1.3, baseline=(current bounding box.center)]
		\foreach \i in {0,0.25,0.5,0.75,1}{
			\draw [lightgray, thick, line cap=round] (0,\i)--(1,\i);
		}
		\foreach \i in {0,0.25,0.5,0.75,1}{
			\draw [lightgray, thick, line cap=round] (\i,0)--(\i,1);
		}
		\foreach \i in {0,.25,.5}{
			\draw[fill=black] (0.33333+\i,0.58333-\i) circle (1pt);
			\draw[fill=black] (0.41667+\i,0.66667-\i) circle (1pt);
		}
		%
		\draw (1.8,0.5)--(2.7,0.5) node [above=2pt, below, midway] {input} node [above, midway] {$1\ 2\ \cdots\ n$};
		\draw (-1.7,0.5)--(-0.8,0.5) node [above=2pt, below, midway] {output};
		\draw[->] (1.6,0.5)--(1.2,0.5); 
		\draw[->] (-0.2,0.5)--(-0.6,0.5);   
		\draw (1,1.25) to [out=0, in=180] (1.6, 0.5);
		\draw (1,1.25)--(0,1.25);
		\draw[->] (0,1.25) to [out=180, in=0] (-0.6, 0.5);
		\node [rotate=-45] at (0.13,0.87) {{\footnotesize $\dots$}};
	\end{tikzpicture}
\end{tabular}
\end{center}
\caption{The $\F_\oplus$- and $\F_\ominus$-machines}
\label{fig-fib-classes}
\end{figure}
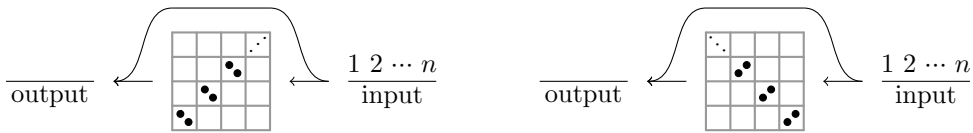

It is known that $\F_\oplus = \Av(231, 312, 321)$ and $\F_\ominus = \Av(123, 132, 213)$. Thus Theorem~\ref{thm-basis} implies that the $\F_\oplus$-machine generates the class $\Av(4231,$ $4312,$ $4321)$, while the $\F_\ominus$-machine generates the class $\Av(4123,$ $4132,$ $4213)$. Note that these classes are both subclasses of Schr\"oder classes considered in the previous section.


\subsection{The $\F_\oplus$-Machine}

In this subsection we consider the class $\Av(4231,4312,4321)$ generated by the $\F_\oplus$-machine. 

\begin{theorem}
The permutation class $\Av(4231,4312,4321)$ generated by the $\F_\oplus$-machine has a generating function which is algebraic of degree 4 and growth rate
\[
	\frac{\left(3-\sqrt{5}\right)\left(7 + 3\sqrt{5} + 2\sqrt{22+10\sqrt{5}}\right)}{4},
\]
approximately $5.16207$.
\end{theorem}

\begin{proof}
We start by crafting an automaton to represent the $\F_\oplus$-machine, similar to that in Figure~\ref{fig-Av21-automata} for the $\Av(21)$-machine. However, this automaton is more complicated than any of the corresponding automata for the Catalan and Schr\"oder classes due to one important fact: the number of places where we can push the next element into the machine can vary between $1$ and $2$ (in the machines for the Catalan classes it was always $1$, and in the machines for the Schr\"oder classes, it was always $2$ so long as the machine was non-empty). As such, the automaton for the $\F_\oplus$-machine, shown in Figure~\ref{fig-fib-fsa}, has 5 states: $E$ represents an empty machine, $S_p$ (resp.~$S_n$) represent states in which the rightmost layer has only one entry and pops are permitted (resp.~forbidden) and $D_p$ (resp.~$D_n$) represent states in which the rightmost layer has two entries and pops are permitted (resp.~forbidden).


\begin{figure}[ht!]
	\begin{footnotesize}
	\begin{center}
		\begin{tikzpicture}[shorten >=1pt,node distance=3cm,on grid,auto] 
			\useasboundingbox (-1.5,-3.5) rectangle (7,3.5);
			\node[state,initial] (E)   {$E$}; 
			\node[state] (Sp) [above right=of E] {$S_p$}; 
			\node[state] (Sn) [below right=of E] {$S_n$}; 
			\node[state] (Dp) [right=of Sp] {$D_p$};
			\node[state] (Dn) [right=of Sn] {$D_n$};
			\path[->] 
				(E) edge node[swap, sloped, anchor=south] {push} (Sn)
					edge [out=270, in=230, loop] node[anchor=north] {bypass} ()
				(Sp) edge node[sloped, anchor=south] {pop if $=1$} (E)
					edge [loop above] node {bypass} ()
					edge [loop left] node {pop if $>1$} ()
					edge [out=280, in=80] node[sloped, anchor=south] {push} (Sn)
					edge node[sloped, anchor=south] {push} (Dn)
				(Sn) edge [loop below] node {push} ()
					edge [out=100, in=260] node[sloped, anchor=south, rotate=180] {bypass} (Sp)
					edge [out=-10, in=190] node[sloped, anchor=north] {push} (Dn)
					edge [out=-45, in=-45, looseness=2] node[sloped, anchor=north] {push} (Dp)
				(Dn) edge [out=170, in=10] node[sloped, anchor=south] {push} (Sn)
					edge node[sloped, anchor=south, rotate=180] {bypass} (Dp)
				(Dp) edge [loop right] node {bypass} ()
					edge [loop above] node {pop if $>2$} ()
					edge node[sloped, anchor=south] {pop if $=2$} (Sp);
		\end{tikzpicture}
	\end{center}
	\end{footnotesize}
	\caption{An automaton representing the $\F_\oplus$-machine}
	\label{fig-fib-fsa}
\end{figure}
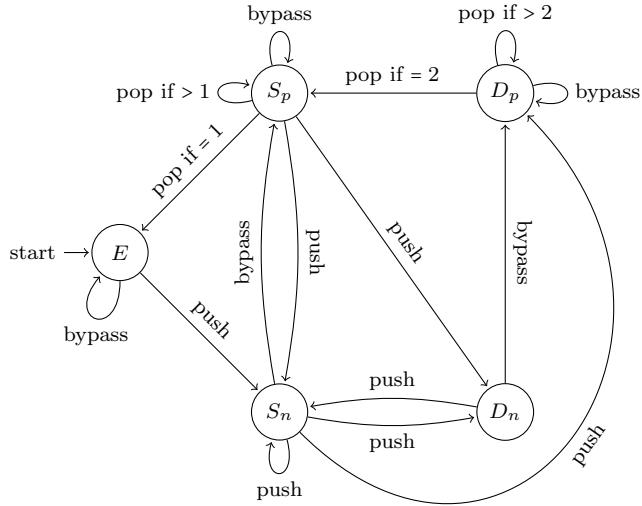

Let $E(x)$, $S_p(x,u)$, $S_n(x,u)$, $D_p(x,u)$, $D_n(x,u)$ be the generating functions that track states as described above, such that $x$ counts the number of entries that have been output (via pops and bypasses) and $u$ counts the number of entries in the machine \emph{excluding the rightmost layer}. The automaton in Figure~\ref{fig-fib-fsa} translates to the following system of functional equations.
\begin{align*}
	E(x) &= 1 + xE(x) + xS_p(x,0)\\
	S_p(x,u) &= x(S_n(x,u) + S_p(x,u)) + \frac{x}{u}\left(S_p(x,u) - S_p(x,0)\right) + x\left(D_p(x,0)\right)\\
	S_n(x,u) &= E(x) + u(S_n(x,u) + S_p(x,u)) + u^2(D_n(x,u) + D_p(x,u))\\
	D_p(x,u) &= x(D_n(x,u) + D_p(x,u)) + \frac{x}{u}\left(D_p(x,u) - D_p(x,0)\right)\\
	D_n(x,u) &= S_n(x,u) + S_p(x,u)
\end{align*}

%

Note that because, for example, in $S_p(x,u)$ the variable $u$ tracks the contents of the machine not considering the rightmost layer, the generating function $S_p(x,0)$ represents states with only a single entry in the machine.

We can use the first, third, and fifth equation to eliminate $E(x)$, $S_n(x,u)$, and $D_n(x,u)$ from the system, leaving
\begin{align*}
	S_p(x,u) &= \frac{x(1-u^2)}{u(1-u-u^2)}S_p(x,u)-\frac{1-x-u-u^2+xu^2}{u(1-x)(1-u-u^2)}S_p(x,0)\\
	& \qquad\qquad	 + \frac{xu^2}{1-u-u^2}D_p(x,u) + xD_p(x,0) + \frac{x}{(1-x)(1-u-u^2)}\\[10pt]
	D_p(x,u) &= \frac{x}{1-u-u^2}S_p(x,u) + \frac{x^2}{(1-x)(1-u-u^2)}S_p(x,0)\\
	& \qquad\qquad  + \frac{x(1-2u^2)}{u(1-u-u^2)}D_p(x,u) - \frac{x}{u}D_p(x,0) + \frac{x}{(1-x)(1-u-u^2)}
\end{align*}
Solving the second equation for $D_p(x,u)$ and substituting into the first leaves a single equation in terms of $S_p(x,u)$, $S_p(x,0)$, and $D_p(x,0)$. Solving that equation for $S_p(x,u)$ and clearing denominators gives
\begin{equation}
	K(x,u)S_p(x,u) = R_1(x,u)S_p(x,0) + R_2(x,u)D_p(x,0) + R_3(x,u)
	\tag{$\dagger$}\label{equation:kernel}
\end{equation}
with
\begin{align*}
	K(x,u) &= (x-1)({u}^{4}-3{u}^{3}x+{u}^{2}{x}^{2}+{u}^{3}-u{x}^{2}-{u}^{2}+2xu-{x}^{2}) \\
	R_1(x,u) &= -x ( {u}^{3}x-{u}^{2}{x}^{2}-{u}^{3}+2{u}^{2}x+u{x}^{2}-{u}^{2}-xu+{x}^{2}+u-x)
\\
	R_2(x,u) &= xu( x-1 )( {u}^{3}-{u}^{2}x+{u}^{2}-u+x )
\\
	R_3(x,u) &= - xu(xu-u+x).
\end{align*}

Theorem 2 of Bousquet-M\'elou and Jehanne~\cite{bousquet-melou:polynomial-equa:} confirms that there are four fractional power series $u_1(x), \ldots, u_4(x)$, counted with multiplicity, such that $K(x,u_i(x)) = 0$. By checking initial terms, we find in fact that the $u_i(x)$ are distinct. Moreover, substituting any $u_i(x)$ into Equation \eqref{equation:kernel} produces a new equation
\[
	0 = R_1(x,u_i(x))S_p(x,0) + R_2(x,u_i(x))D_p(x,0) + R_3(x,u_i(x)).
\]
Combining the equation above for $i=1,2$ with the two equations $K(x,u_1(x))=0$ and $K(x,u_2(x))=0$ gives a system of four polynomial equations from which we might hope to eliminate the variables $u_1$, $u_2$, and $D_p(x,0)$. This elimination can be performed efficiently through the use of Gr\"obner bases\footnote{One must also add an additional equation $X\cdot (u_1-u_2)-1$ to force the distinctness of $u_1$ and $u_2$.}. As a result, we find that
\[
	0 = x\cdot (1+S_p(x,0)) \cdot A(x,S_p(x,0)),
\]
where
\begin{align*}
	A(x,S_p(x,0)) &= (2x^3+8x^2-x)S_p(x,0)^4	-(x^4+3x^3-58x^2+19x-1)S_p(x,0)^3\\
	& \qquad +(3x^4-30x^3+130x^2-56x+7)S_p(x,0)^2\\
	& \qquad -(x^4+3x^3-58x^2+19x-1)S_p(x,0)\\
	& \qquad +(2x^3+8x^2-x).
\end{align*}
Hence, $A(x,S_p(x,0))$ is the minimal polynomial for the algebraic generating function $S_p(x,0)$. One can combine this minimal polynomial with the equation $E(x) = 1 + xE(x) + xS_p(x,0)$ (for example, with a resultant calculation) to find the minimal polynomial for $E(x)$:
\begin{align*}
	(2x^2+8x-1)E(x)^4
	+(x^3+4x^2-46x+5)E(x)^3\\
	+(3x^3-21x^2+94x-9)E(x)^2\\
	+(x^3+12x^2-82x+7)E(x)\\
	+3x^2+26x-2
	& = 0.
\end{align*}
At this point Maple can explicitly solve for $E(x)$ and inspection reveals that the growth rate of the class is as stated in the theorem.
\end{proof}

The enumeration of this class is given by \OEIS{A257561}.

%


\subsection{The $\F_\ominus$-Machine}

The $\F_\ominus$-machine differs from the $\F_\oplus$-machine because in the $\F_\ominus$-machine a pop can reduce the size of the leftmost layer --- which in this case is the layer we might push into --- thereby opening up more possibilities for the next push and forcing us in some sense to remember the size of the layer to its right (in case a pop empties the leftmost layer).

\begin{theorem}
The class $\Av(4123,4132,4213)$ generated by the $\F_{\ominus}$-machine has generating function $s(x)$ which is algebraic of degree 3 and its growth rate is
\[
	\text{\footnotesize $\frac{67240 + (779\sqrt{57}-1927)(1502 + 342\sqrt{57})^{1/3} - (19\sqrt{57} - 457)(1502 + 342\sqrt{57})^{2/3}}{40344}$},
\]
approximately $5.21914$.
\end{theorem}

\begin{proof}
Because of the difference between the $\F_{\ominus}$- and $\F_{\oplus}$-machines noted above, a direct approach to computing the generating function of $\Av(4123,$ $4132,$ $4213)$ based around an iterative scheme for accounting for the generation sequences is not feasible. For this reason we construct a context-free grammar instead. 

Let $W_n$ be the language of words (tracking states of the $\F_\ominus$-machine) that begin from a state where the leftmost layer is a single entry with no immediate legal pop and end with the same entry alone in the leftmost layer with a pop now allowed. Similarly, $R_n$ will be the language of words beginning in a state where the leftmost layer contains two entries with no immediate legal pop and ending with the same two entries in the leftmost layer with a pop now allowed. Let $W_p$ (resp., $R_p$) be the language of words that start with a single entry (resp., two entries) in the leftmost layer with a legal pop allowed and end with the same single entry (resp., two entries) in the leftmost layer with a legal pop allowed.

These definitions yield the following context-free grammar for legal operation sequences in the $\F_\ominus$-machine.
%
\[
	\begin{array}{cccclllll}
		S &\longrightarrow & \epsilon & \; |  & xS & \; | \; & (+w)W_n(-w)S &\\
		W_p &\longrightarrow & \epsilon & \; |  & xW_p & \; | \; & (+w)W_n(-w)W_p & \; | \; & (+r)R_n(-r)W_p\\
		W_n &\longrightarrow & & &  xW_p & \; | \; & (+w)W_n(-w)W_p & \; | \; & (+r)R_n(-r)W_p\\
		R_p &\longrightarrow & \epsilon  & \; |  & xR_p &  \; | \; & (+w)W_n(-w)R_p \\
		R_n &\longrightarrow & & & xR_p  & \; | \; & (+w)W_n(-w)R_p
	\end{array}
\]
The nonterminals $S$, $W_p$, and $R_p$ each have a production rule to $\epsilon$ because the starting condition satisfies the ending condition for each of these languages, whereas this is not true for $W_n$ and $R_n$. The production rules of the form $T \longrightarrow xT$ represent a bypass operation. As popping is always permitted after a bypass, the bypass is always followed by a state in which popping is legal (e.g., $W_n \longrightarrow xW_p$).

The remaining production rules correspond to pushing an element in a new layer (represented by $(+w)$) or adding an entry to an existing layer of size one (represented by $(+r)$), then performing an appropriate sequence $W_n$ or $R_n$, then popping the entry added earlier (represented by $(-w)$ or $(-r)$). Lastly, each of these production rules ends by allowing a repeated occurrence of either $W_p$ in the case where the production symbol is $W_p$ or $W_n$ or of $R_p$ in the case where the production symbol is $R_p$ or $R_n$.

This context-free grammar is unambiguous because in every rule the start symbols of the various cases are distinct. Hence, we can translate the grammar to equations, replacing $(-w)$ and $(-r)$  with $x$ to keep track of pop operations, and ignoring $(+w)$ and $(+r)$ because we do not need to keep track of pushes. This yields the following system.
\begin{align*}
	s &= 1 + xs + xw_ns\\
	w_p &= 1 + xw_p + xw_nw_p + xr_nw_p\\
	w_n &= xw_p + xw_nw_p + xr_nw_p\\
	r_p &= 1 + xr_p + xw_nr_p\\
	r_n &= xr_p + xw_nr_p
\end{align*}
Another Gr\"obner basis calculation reveals that $s$ satisfies
\[
	1 + (x-1)s(x) - xs(x)^2 + xs(x)^3.
\]
This implies that the growth rate of the class is as stated in the theorem.
\end{proof}

The enumeration of the class is given by \OEIS{A106228}.

%

\section{Machines for restricted classes}
\label{sec-finite-bounded}

As the reader may have noticed, the analysis of $\C$-machines typically depends on the  specific structure of the class $\C$ itself. However, by placing restrictions on the characteristics of the class $\C$ it is possible to establish some general characteristics of the classes corresponding to the associated $\C$-machine. In this section we explore three such restrictions for which we are able to establish general results: to finite, bounded, and polynomial classes.

\subsection{Finite classes}

First we consider the case where $\C$ is a finite class. Following Albert, Atkinson, and Ru\v{s}kuc~\cite{albert:regular-closed-:}, we say that the \emph{rank} of the entry $\pi(i)$ is the number of entries below it and to its right. When $\C$ is finite, the class of permutations that can be generated by the $\C$-machine necessarily has bounded rank. Moreover, because every finite class has a finite basis (an easy consequence of the Erd\H{o}s--Szekeres Theorem), the class of permutations that can be generated by the $\C$-machine has a finite basis, and the results of \cite{albert:regular-closed-:} imply that this class has a rational generating function\footnote{Indeed, these classes fall under the purview not only of the rank encoding, but also of the finitely labeled generating trees of Vatter~\cite{vatter:finitely-labele:} and the insertion encoding of Albert, Linton, and Ru\v{s}kuc~\cite{albert:the-insertion-e:}.}.

\begin{theorem}
\label{thm-finite-machine}
If $\C$ is a finite class then the class of permutations that can be generated by the $\C$-machine has a rational generating function.	
\end{theorem}

More in the spirit of the ``$\C$-machine approach'' an alternative way to prove this theorem would be to note that for a finite class $\C$ there are only finitely many states that the $\C$-machine can take (at most twice the total number of permutations in $\C$ allowing for the ``can't pop'', ``can pop'' distinction). Therefore, after choosing an alphabet that allows us to represent all the different ways that a new maximum element can be added to the machine, the $\C$-machine can be thought of as a finite state automaton. The conclusion of the theorem follows immediately.


\subsection{Polynomial classes}

We next consider the case where $|\C_n|$ is bounded by a polynomial (in $n$), in which case we call $\C$ a \emph{polynomial class}. Kaiser and Klazar~\cite{kaiser:on-growth-rates:} established two significant results regarding polynomial classes. First, they showed that polynomial classes are actually enumerated by polynomials for sufficiently large $n$ (i.e., they are not just polynomially bounded). Second, they showed that if the enumeration of a class is ever less than the $n$th \emph{combinatorial Fibonacci number} (defined by $F_0=F_1=1$ and $F_n=F_{n-1}+F_{n-2}$) then the class is a polynomial class. This second statement is referred to as the Fibonacci Dichotomy. Later, Huczynska and Vatter~\cite{huczynska:grid-classes-an:} reproved the Fibonacci Dichotomy using what are known as grid classes and gave an explicit characterization of polynomial classes. While we do not need to appeal to the details of this characterization, we do require the following fact that follows from it.

\begin{proposition}
\label{prop-poly-fin-basis}
Every polynomial class is finitely based.	
\end{proposition}

Proposition~\ref{prop-poly-fin-basis} is explicitly proved in the conclusion of Huczynska and Vatter~\cite{huczynska:grid-classes-an:} and also follows from the later and more general Vatter~\cite[Theorem 6.2]{vatter:small-permutati:}.

Our result about polynomial classes requires one further notion. Inspired by Wilf's infuential \emph{Monthly} article ``What is an answer?''~\cite{wilf:what-is-an-answ:}, Zeilberger~\cite{Zeilberger:Enumerative-and:} defined a \emph{Wilfian formula} for the sequence $\{a_n\}$ to be a polynomial-time (in $n$) algorithm that computes $a_n$. For example, an algebraic generating function can easily be converted into a Wilfian formula (one needs only to compute derivatives), but many sequences that do not have algebraic generating functions still have Wilfian formulas (e.g., the Catalan numbers modulo $2$).

\begin{theorem}
\label{thm-poly-machine}
If $\C$ is a polynomial class then the class of permutations that can be generated by the $\C$-machine has a Wilfian formula.
\end{theorem}
\begin{proof}
Let $\C$ be a polynomial class and choose a polynomial $c(n)$ such that $|\C_n|\le c(n)$ for all $n$. By Proposition~\ref{prop-poly-fin-basis}, $\C=\Av(B)$ for a finite set $B$. Let $m$ denote the length of the longest basis element of $B$. Thus we can determine whether a permutation of length $n$ lies in $\C$ in time $b(n)=|B|{n\choose m}$. We seek to show that there is a polynomial $p(n)$ such that we can determine the number of permutations of length $n$ that can be generated by the $\C$-machine in time at most $p(n)$.

To accomplish this, we create an automaton that has two states for each permutation of length at most $n$ in $\C$. Of these two states, one corresponds to the ``can pop'' condition and the other to the ``can't pop'' condition, while the permutation associated to the state records the order isomorphism type of the contents of the machine. We can build this automaton by working up from the states corresponding to the empty permutation by considering all possible pushes, pops (if the ``can pop'' condition is true for that state), and bypasses. For each state whose corresponding permutation has length $k$, there are at most $k+3$ such actions. Pops and bypasses are trivial to analyze, while for each possible push we can determine if the push leads to a permutation in $\C$ in time $b(k+1)$. Therefore we can construct this automaton in time at most
\[
	\sum_{k=0}^{n-1} (k+3) \mkern1.0mu b(k+1) \mkern1.0mu  c(k), 
\]
which is a polynomial of degree at most $2+m+\deg c$. To compute $|\C_n|$ from this automaton we simply count the number of closed walks beginning and ending at the state corresponding to the empty permutation that contain a total of $n$ pops and bypasses. As the automaton has only a polynomial number of states, the number of these walks can be computed in polynomial time.
\end{proof}

The argument above carries through almost directly when $\C$ is not quite a polynomial class, but instead the sum closure of a polynomial class. For example the permutations in $\Av(231,321)$ are all sums of permutations of the form $k12\cdots(k-1)$. The corresponding machine is analyzed in the next section. 

Needless to say, the algorithm described in the proof of Theorem~\ref{thm-poly-machine} should not be implemented. To obtain a more practical algorithm for enumerating these $\C$-machines, one would want to implement a dynamic programming algorithm exploiting the specific structure of $\C$. We present several examples of this in Section \ref{sec-non-D-finite}.

\subsection{Bounded classes}

We conclude this section with the consideration of \emph{bounded classes}: those classes $\C$ for which there exists an integer $c$ such that $|\C_n| \leq c$ for all $n \geq 0$. Obviously bounded classes are a special case of polynomial classes, but because our result is stronger we must describe the structure of bounded classes in more detail. In doing so we follow Homberger and Vatter~\cite{homberger:on-the-effectiv:}.

An \emph{interval} in a permutation is a sequence of contiguous entries whose values form an interval of natural numbers. A \emph{monotone interval} is an interval in which the entries are monotone (increasing or decreasing).  Given a permutation $\sigma$ of length $m$ and nonempty permutations $\alpha_1,\dots,\alpha_m$, the \emph{inflation} of $\sigma$ by $\alpha_1,\dots,\alpha_m$ is the permutation $\pi=\sigma[\alpha_1,\dots,\alpha_m]$ obtained by replacing each entry $\sigma(i)$ by an interval that is order isomorphic to $\alpha_i$, while maintaining the relative order of the intervals themselves.  For example,
\[
	3142[1,321,1,12]=6\ 321\ 7\ 45.
\]

We define a \emph{peg permutation} to be a permutation where each entry is decorated with a $+$, $-$, or $\bullet$, such as
\[
\tilde{\rho}=\d3\m1\d4\p2
\]
The \emph{grid class} of the peg permutation $\tilde{\rho}$, denoted $\Grid(\tilde{\rho})$, is the set of all permutations that may be obtained by inflating $\rho$ (the underlying, non-decorated version of $\tilde{\rho}$) by monotone intervals of type determined by the signs of $\tilde{\rho}$: $\rho(i)$ may be inflated by an increasing (resp., decreasing) interval if $\tilde{\rho}(i)$ is decorated with a $+$ (resp., $-$) while it may only be inflated by a single entry (or the empty permutation) if $\tilde{\rho}(i)$ is dotted.  Thus if $\pi\in\Grid(\tilde{\rho})$ then its entries can be partitioned into monotone intervals which are ``compatible'' with $\tilde{\rho}$.

Given a set $\tilde{G}$ of peg permutations, we denote the union of their corresponding grid classes by
\[
\Grid(\tilde{G})=\bigcup_{\tilde{\rho}\in\tilde{G}} \Grid(\tilde{\rho}).
\]

In their characterization of polynomial classes, Huczynska and Vatter~\cite{huczynska:grid-classes-an:} proved that every polynomial class is contained in $\Grid(\tilde{\rho})$ for a single peg permutation $\tilde{\rho}$. From this and the work of Albert, Atkinson, Bouvel, Ru\v{s}kuc, and Vatter on atomic geometric grid classes~\cite[Theorem 10.3]{albert:geometric-grid-:}, the following result follows.

\begin{theorem}
\label{thm-poly-grids}
For every polynomial class $\C$ there is a finite set $\tilde{G}$ of peg permutations such that $\C=\Grid(\tilde{G})$.	
\end{theorem}

The containment relation on $\mathbb{N}^m$ 
is  a partial order. Thus we may define downsets (sets closed downward under containment) and upsets of vectors. The intersection of a downset and an upset is referred to as a \emph{convex set}.

We say that $\vec{v}$ \emph{fills} the peg permutation $\tilde{\rho}$ if $\vec{v}(i) = 1$ whenever $\tilde{\rho}(i)$ is decorated with a $\bullet$ and $\vec{v}(i) \geq 2$ whenever $\tilde{\rho}(i)$ is decorated with a $+$ or $-$.
Given any peg permutation $\tilde{\rho}$ of length $m$ and a set of vectors $\mathcal{V}\subseteq\mathbb{P}^m$ that fill $\tilde{\rho}$, we define
\[
	\tilde{\rho}[\mathcal{V}]=\bigcup_{\vec{v}\in\mathcal{V}} \tilde{\rho}[\vec{v}].
\]
We now have all the terminology and notation to state the relevant structure theorem.

\begin{theorem}[Homberger and Vatter~\cite{homberger:on-the-effectiv:}]
\label{thm-polynomial-main}
For every polynomial permutation class $\C$ there is a finite set $\tilde{G}$ of peg permutations, each associated with its own convex set $\mathcal{V}_{\tilde{\sigma}}$ of vectors of positive integers which fill it, such that $\C$ can be written as the disjoint union
\[
	\C=\biguplus_{\tilde{\rho}\in\tilde{G}} \tilde{\rho}[\mathcal{V}_{\tilde{\rho}}].
\]	
\end{theorem}

This theorem is more useful than Theorem \ref{thm-poly-grids} for the description of an enumerative scheme leading to Theorem \ref{thm-bounded-machine} below precisely because the union it describes is disjoint.

We establish our result about bounded classes using \emph{counter automata}, which are finite state automata with the additional ability to store a single nonnegative integer called a \emph{counter}. When determining which transition to take, a counter automaton is allowed to check if the value of the counter is $0$, and during each transition the value of the counter may be incremented or decremented by $1$. Equivalently, for any fixed positive integer $N$ and all $n$ satisfying $0\le n\le N$, a counter automaton is allowed to check if the value of the counter is equal to $n$ and is allowed to increase or decrease the counter by $n$. Deterministic counter automata are a proper subset of deterministic pushdown automata and therefore the languages they accept have algebraic generating functions (see Droste, Kuich, and Vogler~\cite[Chapter 7]{:Handbook-of-wei:}).

\begin{theorem}
\label{thm-bounded-machine}
If $\C$ is a bounded class then the class of permutations that can be generated by the $\C$-machine has an algebraic generating function.
\end{theorem}
\begin{proof}
	Suppose $\C$ is a bounded class and let $\tilde{G}$ and the convex sets $\V_{\tilde{\rho}}$ for each $\tilde{\rho}\in\tilde{G}$ be as in the statement of Theorem~\ref{thm-polynomial-main}. We build a counter automaton whose states represent the subpermutation in the container of the $\C$-machine at any point in time. However, as $\C$ contains infinitely many permutations (otherwise it would fall under the purview of Theorem~\ref{thm-finite-machine}) and a standard counter automaton must have a finite number of states, some compression is necessary.
	
	Each $\tilde{\rho}$ comes equipped with a convex set $\V_{\tilde{\rho}}$ of vectors in $\mathbb{N}^{|\tilde{\rho}|}$. Only one component of these vectors is allowed to grow unboundedly as otherwise the class $\C$ would not be bounded. For each $\tilde{\rho}\in\tilde{G}$ let $M_{\tilde{\rho}}$ denote the maximum value of all \emph{other} components  for $\vec{v}\in\V_{\tilde{\rho}}$ and define
	\[
		M = \max(\{M_{\tilde{\rho}} : \tilde{\rho} \in \tilde{G}\}).
	\]
	That is, $M$ is the maximum of all second-largest components over all $\vec{v} \in \V_{\tilde{\rho}}$ and $\tilde{\rho} \in \tilde{G}$.
		
	Any state of the $\C$-machine in which the container holds a subpermutation $\tilde{\rho}[\vec{v}]$ with $\vec{v}(i) \leq M$ for all $i$ is simply represented by a state of the counter automaton labeled $\tilde{\rho}[\vec{v}]$. Any state of the $\C$-machine in which the container holds a subpermutation $\tilde{\rho}[\vec{v}]$ with some $\vec{v}(i) > M$ is represented by a state of the counter automaton labeled
		\[
			\tilde{\rho}[\vec{v}(1), \ldots, \vec{v}(i-1), *, \vec{v}(i+1), \ldots, \vec{v}(|\tilde{\rho}|)].
		\]
	Here the $*$ symbol represents an inflation of size at least $M+1$, and it is this parameter that the counter keeps track of by storing the value $\min\{0, \vec{v}(i)-M\}$.
		
	Next we split each state described above into two copies: one labeled ``can pop'' and one labeled ``can't pop''. We add to this a state labeled $\epsilon$ to account for the empty machine which is both the start state and the unique accepting state. The transitions between each pair of states are readily computed by examining the allowed pushes, pops, and bypasses. Transitions to states with no `$*$' marker must set the counter at $0$, while transitions to states with a `$*$' marker may or may not change the counter (they can also, of course, change the underlying $\tilde{\rho}$).
	
	
	The counter automaton constructed above accepts all valid push/pop sequences that leave the container of the $\C$-machine empty. If transitions are weighted so that those corresponding to bypasses and pops have weight $x$ and those corresponding to pushes have weight $1$, then the weighted generating function counting accepting paths of length $n$ is equal to the generating function for the class generated by the $\C$-machine.
\end{proof}

As with all the results of this section, note that Theorem~\ref{thm-bounded-machine} represents only a sufficient condition for algebraicity. In particular, it does not apply to any of the Schr\"oder machines which nevertheless generate classes with algebraic generating functions.

\section{Potentially Non-D-Finite Classes}
\label{sec-non-D-finite}

Here we present four permutation classes for which, despite the fact that they can be generated by fairly simple $\C$-machines, we do not know (and cannot even conjecture) their generating functions. Indeed, while we can implement the dynamic programming approach hinted at in the proof of Theorem~\ref{thm-poly-machine} to obtain many terms in the counting sequence of these classes (5,000 in the first case we present), we cannot fit a D-finite generating function to any of them. The first case we present has three basis elements of length four while the three following it are so-far-unenumerated ``2$\times$4 classes''. We present the first example in detail but provide only sketches of the (very similar) arguments for the remaining examples.

\subsection{$\Av(4123, 4231, 4312)$}
\label{subsec-Av-4123-4231-4312}

Theorem~\ref{thm-basis} implies that the class $\Av(4123,4231,4312)$ is generated by the $\Av(123,231,312)$-machine. The members of $\Av(123,231,312)$ can be drawn as shown below, where the labeling of entries with an $\textsf{a}$ or $\textsf{b}$ is for the subsequent analysis.

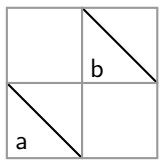
\begin{center}
	\begin{tikzpicture}[scale=2, baseline=(current bounding box.center)]
		\draw [thick, line cap=round] (0,0.5)--(0.5,0);
		\draw [thick, line cap=round] (0.5,1)--(1,0.5);
		\foreach \i in {0,0.5,1}{
			\draw [lightgray, thick, line cap=round] (0,\i)--(1,\i);
			\draw [lightgray, thick, line cap=round] (\i,0)--(\i,1);
		}
		\node at (.1, .1) {\textsf{a}};
		\node at (.6, .6) {\textsf{b}};
	\end{tikzpicture}	
\end{center}

When the container is empty we may only push an \textsf{a} entry. When it contains a decreasing permutation (all of whose entries are viewed as \textsf{a} entries), we may push either an \textsf{a} or a \textsf{b} entry. After pushing a \textsf{b} entry we may only push \textsf{b} entries until we have popped all of the \textsf{a} entries, at which point all current \textsf{b} entries become \textsf{a} entries.

We represent the states of the $\Av(123,231,312)$-machine by triples $(a,b,P)$ where $a$ and $b$ are the number of \textsf{a} and \textsf{b} entries respectively, and $P$ is either \texttt{T(rue)} or \texttt{F(alse)}, depending on whether popping is allowed. The preceding discussion shows:

\begin{proposition}
The transition rules for the $\Av(123,231,312)$-machine are:
\begin{eqnarray*}
	(0, 0, \true)
		&\longrightarrow&
		\{(1, 0, \false), (0, 0, \true)\},\\
	(a, 0, \false)
		&\longrightarrow&
		\{(a+1, 0, \false), (a, 1, \false), (a, 0, \true)\},\\
	(a, 0, \true)
		&\longrightarrow&
		\{(a+1, 0, \false), (a, 1, \false), (a, 0, \true), (a-1, 0, \true)\},\\
	(a, b, \false)
		&\longrightarrow&
		\{(a, b+1, \false), (a, b, \true)\},\\
	(a, b, \true)
		&\longrightarrow&
		\{(a, b+1, \false), (a, b, \true), (a-1, b, \true)\}\quad\mbox{(for $a \ge 2$)},\\
	(1, b, \true)
		&\longrightarrow&
		\{(1, b+1, \false), (1, b, \true), (b, 0, \true)\},
\end{eqnarray*}
where $a, b \ge 1$ unless stated otherwise.
\end{proposition}

These transition rules can be adapted to a dynamic programming algorithm, which can be used to compute the first 5,000 terms of the enumeration in a moderate amount of time. The enumeration of this class is \OEIS{A257562}.

One can also derive a functional equation for the generating function of this class using these transition rules. Define an $A$ state to be one in which there are no \textsf{b} entries and a $B$ state to be one in which there are \textsf{b} entries (and therefore, also \textsf{a} entries). We require that popping is always permitted at the beginning of a $B$ state (we explain this in more detail below). The empty state is considered an $A$ state, and $A$ is also the start state. 

Let $A(a,x)$ be the generating function in which the coefficient of $a^kx^n$ counts the number of ways to reach an $A$ state with $k$ entries labelled \textsf{a} and $n$ entries output so far. Let $B(a,b,x)$ be the generating function in which the coefficient of $a^k b^\ell x^n$ counts the number of ways to reach a $B$ state with $k-1$ entries labelled \textsf{a}, $\ell$ entries labelled \textsf{b}, and $n$ entries output so far. As $B$ tracks one fewer than the number of \textsf{a} entries, it follows that $B(0,b,x)$ enumerates the $B$ states with exactly one \textsf{a} entry.

\begin{proposition}
The generating functions $A(a,x)$ and $B(a,b,x)$ that describe the operation of the $\Av(123,231,312)$-machine satisfy the functional equations
\begin{align*}
	A(a,x) &= 1 + \frac{x}{a}(A(a,x) - A(0,x)) + aA(a,x) + xB(0,a,x),\\
	B(a,b,x) &= \frac{bx}{a(1-b)(1-x)}(A(a,x) - A(0,x)) + \frac{bx}{(1-b)(1-x)}B(a,b,x)\\
	&\qquad\qquad + \frac{x}{a(1-x)}(B(a,b,x) - B(0,b,x)).
\end{align*}
\end{proposition}

\begin{proof}
The $a$ in the denominator in the first term in $B(a,b,x)$ accounts for the fact that $B(a,b,x)$ tracks one fewer than the number of $A(a,x)$.

An $A$ state is reached from another $A$ state either by popping an \textsf{a} entry (if there is one) or by pushing an \textsf{a} entry. We ignore bypasses in this viewpoint; if we pop an \textsf{a} entry while there are no \textsf{b} entries, then that \textsf{a} entry could have been treated as a bypass. An $A$ state is reached from a $B$ state only by popping the last \textsf{a} entry in a $B$ state with a single \textsf{a} entry. Therefore, the generating function $A(a,x)$ satisfies
\[
	A(a,x) = 1 + \frac{x}{a}(A(a,x)-A(0,x)) + aA(a,x) + xB(0,a,x).
\]
The term $1$ accounts for the start state. The term $(x/a)(A(a,x)-A(0,x))$ accounts for popping an \textsf{a} entry if there is one. The term $aA(a,x)$ accounts for pushing an \textsf{a} entry. Lastly, the term $xB(0,a,x)$ accounts for popping the final \textsf{a} from a $B$ state with exactly one \textsf{a} entry, forcing all \textsf{b} entries to become \textsf{a} entries. It is important here that we assumed popping is always permitted in a $B$ state.

We can reach a $B$ state from an $A$ state with at least one \textsf{a} by pushing a \textsf{b}. However, we do not want a term $b(A(a,x)-A(0,x))$ in the functional equation for $B(a,b,x)$ because the state resulting from pushing a single \textsf{b} does not allow for popping --- this would violate our uniqueness conventions, because the entry that can be popped is the leftmost \textsf{a} entry which we could have popped before pushing the \textsf{b} entry. For this reason, we consider more elaborate transitions to $B$ states: instead of pushing a single \textsf{b} entry, we push a sequence of $\textsf{b}$ entries followed by at least one bypass (accounted for by the first term in $B(a,b,x)$ above) while a pop of an $\textsf{a}$ entry may be followed by any number of bypasses (accounted for by the second term above).
\end{proof}

One can in principle iterate this functional equation starting with $A_0(a,x) = 1$ and $B_0(a,b,x) = 0$ to obtain terms of $A(0,x)$. It is clear from the description of pushing and popping that after $2n$ iterations the coefficient of each $x^i$ for $0 \leq i \leq n$ in the resulting $A_{2n}(0,x)$ will match the coefficient of $x^i$ in $A(0,x)$. However, this is much slower than the dynamic programming approach.

After obtaining 5,000 terms of the sequence enumerating $\Av(4123,4231,4312)$ we apply the method of differential approximation~\cite{guttmann:on-a-new-method:}, an experimental method that uses initial terms of a sequence to non-rigorously approximatate asymptotic growth of the form
\[
	C\gamma^n n^{-1-\alpha}.
\]
For this class, we predict $C\approx 0.01897$, $\gamma\approx 4.46841$, and $\alpha=-1$, so we conjecturally have the approximate asymptotics
\[
	|\Av_n(4123,4231,4312)|\sim 0.01897\cdot (4.46841)^n.
\]
We can approximate these constants to around $100$ decimal places of accuracy. 

A function $f(x)$ is said to be \emph{differentially algebraic} if there exists some $k \geq 0$ and some polynomial $P(x_1, x_2, \ldots, x_{k+2})$ such that
	\[
		P(x, f(x), f'(x), \ldots, f^{(k)}(x)) = 0\tag{$\star$}\label{ADE}
	\]
for all $x$. Equation \eqref{ADE} is called an \emph{algebraic differential equation}. All algebraic and differentially finite (D-finite) generating functions are also differentially algebraic.

We have written a Maple program to use terms of a counting sequence to guess an algebraic differential equation which might be satisfied by the generating function of a given sequence. This program has not been able to guess an algebraic differential equation that might be satisfied by the generating function for $\Av(4123,4231,4312)$, despite searching through all possible algebraic differential equations
\[
	P(x, f(x), f'(x), \ldots, f^{(k)}(x)) = 0
\]
with $x_1$-degree at most $d$, differential order at most $k$, and $(x_2,\ldots,x_{k+2})$-total degree at most $m$ such that
\[
	(d+1){m + k + 1 \choose m} + 5 \leq 5000.
\]
In light of this, we make the following conjecture.

\begin{conjecture}
\label{conj-non-ADE-first}
	The generating function of the class $\Av(4123,4231,4312)$ is not differentially algebraic.
\end{conjecture}

In particular, Conjecture~\ref{conj-non-ADE-first} would imply that the generating function for this class is not D-finite. Note that every subclass of $\Av(123,231,312)$ has bounded enumeration, and thus by Theorem~\ref{thm-bounded-machine} their machines generate classes with algebraic generating functions. Thus it appears that the $\Av(123,231,312)$-machine is a minimal non-algebraic machine.

\subsection{Three $2\times 4$ classes}

There are three so-far unenumerated $2\times 4$ classes which may be generated by $\C$-machines. For full details on these three cases, we refer to the third author's thesis~\cite[Chapter 5]{pantone:structural-anal:}, giving only an outline here. The containers of these machines are shown in Figure~\ref{fig-2x4-containers}. The first two containers, $\Av(123,231)$ and $\Av(123,231)$, may be analyzed using three parameters which record the sizes of the three blocks.

\begin{figure}
\begin{center}
	\begin{tikzpicture}[scale=2, baseline=(current bounding box.center)]
		\draw [thick, line cap=round] (0,1)--(0.33,0.67);
		\draw [thick, line cap=round] (0.33,0.33)--(0.67,0);
		\draw [thick, line cap=round] (0.67,0.67)--(1,0.33);
		\foreach \i in {0,0.33,0.67,1}{
			\draw [lightgray, thick, line cap=round] (0,\i)--(1,\i);
			\draw [lightgray, thick, line cap=round] (\i,0)--(\i,1);
		}
	\end{tikzpicture}	
\quad\quad\quad\quad
	\begin{tikzpicture}[scale=2, baseline=(current bounding box.center)]
		\draw [thick, line cap=round] (0,0.67)--(0.33,0.33);
		\draw [thick, line cap=round] (0.33,1)--(0.67,0.67);
		\draw [thick, line cap=round] (0.67,0.33)--(1,0);
		\foreach \i in {0,0.33,0.67,1}{
			\draw [lightgray, thick, line cap=round] (0,\i)--(1,\i);
			\draw [lightgray, thick, line cap=round] (\i,0)--(\i,1);
		}
	\end{tikzpicture}	
\quad\quad\quad\quad
	\begin{tikzpicture}[scale=2, baseline=(current bounding box.center)]
		\draw [thick, line cap=round] (0,0)--(0.75,0.75);
		
		\foreach \i in {0,0.25,0.5,0.75,1}{
			\draw [lightgray, thick, line cap=round] (0,\i)--(1,\i);
		}
		\foreach \i in {0,0.25,0.5,0.75,1}{
			\draw [lightgray, thick, line cap=round] (\i,0)--(\i,1);
		}
		\useasboundingbox (current bounding box.south west) rectangle (current bounding box.north east);
		
		\draw[fill=black] (0,0.25) circle (1pt);
		\draw[fill=black] (0.25,0.5) circle (1pt);
		\draw[fill=black] (0.5,0.75) circle (1pt);
		\node [rotate=45] at (0.89,0.89) {{\footnotesize $\dots$}};
	\end{tikzpicture}
	\end{center}
\caption{From left to right, the classes $\Av(123, 231)$, $\Av(123, 312)$, and $\Av(231,321)$.}
\label{fig-2x4-containers}
\end{figure}
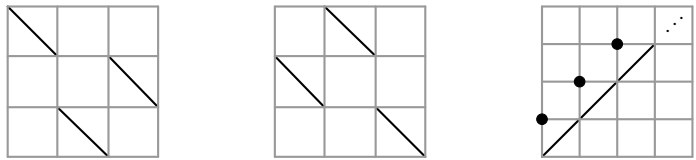

The class $\Av(231,321)$, however, is not a polynomial class; in fact, there are $2^{n-1}$ permutations of length $n$ in this class. As such, the strategy of using a different variable to represent entries in each block is not effective. Instead, we label the entries of the maximum increasing final block (if there are any) by \textsf{a}, we label all entries in the rightmost sum component of the form $1 \ominus (12\cdots\ell)$ by \textsf{b}. The remaining entries of the container are labelled by \textsf{c}. Figure~\ref{fig-Av-231-321-vars} shows three examples. As this assignment implies, we do not need to keep track of the actual shape formed by the \textsf{c} entries, even though they can take many different forms. Once an entry becomes a \textsf{c} entry, it stays a \textsf{c} entry until it is popped.

\begin{figure}
	\begin{footnotesize}
	\begin{center}
		\begin{tikzpicture}[scale=.2, baseline=(current bounding box.center)]
			\foreach [count=\i] \x in {1,2,6,3,4,5,7,11,8,9,10,12,13,14}{
				\draw[fill] (\i, \x) circle (0.2);
			}
			\draw[thick, gray] (.5, .5)  rectangle (7.5,7.5);
			\draw[thick, gray] (7.5, 7.5)  rectangle (11.5,11.5);
			\draw[thick, gray] (11.5, 11.5)  rectangle (14.5,14.5);
			\node at (6.8,1.2) {\textsf{c}};
			\node at (10.8,8.2) {\textsf{b}};
			\node at (13.8,12.2) {\textsf{a}};
		\end{tikzpicture}
		\qquad\qquad
		\begin{tikzpicture}[scale=.2, baseline=(current bounding box.center)]
			\foreach [count=\i] \x in {1,2,6,3,4,5,7,11,8,9,10}{
				\draw[fill] (\i, \x) circle (0.2);
			}
			\draw[thick, gray] (.5, .5)  rectangle (7.5,7.5);
			\draw[thick, gray] (7.5, 7.5)  rectangle (11.5,11.5);
			\node at (6.8,1.2) {\textsf{c}};
			\node at (10.8,8.2) {\textsf{b}};
		\end{tikzpicture}
		\qquad\qquad
		\begin{tikzpicture}[scale=.2, baseline=(current bounding box.center)]
			\foreach [count=\i] \x in {4,1,2,3,5,6,7}{
				\draw[fill] (\i, \x) circle (0.2);
			}
			\draw[thick, gray] (.5, .5)  rectangle (4.5,4.5);
			\draw[thick, gray] (4.5, 4.5)  rectangle (7.5,7.5);
			\node at (3.8,1.2) {\textsf{b}};
			\node at (6.8,5.2) {\textsf{a}};
		\end{tikzpicture}
	\end{center}
	\end{footnotesize}
	\caption{Three examples of the variable assignments in the $\Av(231,321)$-machine.}
	\label{fig-Av-231-321-vars}
\end{figure}
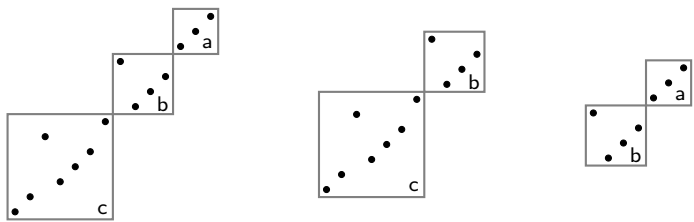

With dynamic programming, we are able to easily compute 1,000 terms in each of first two enumerations, and 600 in the final case, and could extend this enumeration were it of interest. Again we use the method of differential approximation to fit these sequences to asymptotics of the form $C\gamma^n n^{-1-\alpha}$. The data for these three cases is collected in the table below.

\[
	\begin{array}{llccl}
		\text{Class}&\text{Sequence}&\text{$\gamma\approx$}&\text{$\alpha\approx$}&\text{\OEISref}\\\hline
		\Av(4123,4231)&1, 2, 6, 22, 89, 380, 1677, 7566&4.97689&-1&\OEISlink{A165542}\\
		\Av(4123,4312)&1, 2, 6, 22, 89, 382, 1711, 7922&3+2\sqrt{2}&1/2&\OEISlink{A165545}\\
		\Av(4231,4321)&1, 2, 6, 22, 90, 396, 1837, 8864&5.89249&-1&\OEISlink{A053617}\\
		\hline
	\end{array}
\]

As in the previous subsection, we have also used these terms to attempt to conjecture differentiably algebraic generating functions fitting the known data, with no success. This even holds in the second example, despite the consistency of the conjectured values of $\gamma$ and $\alpha$ with simple algebraic generating functions.

\begin{conjecture}
\label{conj-non-ADE-second}
	None of the classes $\Av(4123,4231)$, $\Av(4123,4312)$, or $\Av(4231,4321)$ have differentially algebraic generating functions.
\end{conjecture}

\section{Concluding Remarks}

There are many more permutation classes that could be enumerated --- either obtaining an explicit generating function or generating hundreds or thousands of terms --- with $\C$-machines. While Section~\ref{sec-finite-bounded} initiates the study of the more general theory of how restrictions on a class $\C$ may imply certain properties of the class generated by the $\C$-machine, the four classes considered in Section~\ref{sec-non-D-finite} suggest that extending this classification may require great care.

For instance, we presented the class $\Av(4123,4231,4312)$ generated by the $\Av(123,231,312)$-machine. Recall that the class $\Av(123,231,312)$ is a polynomial class represented by the peg permutation $1^-2^-$ and that we conjecture that the class $\Av(4123,4231,4312)$ does not have a differentially algebraic generating function. One might suspect that the cause of this complicated behavior is the presence of two entries inflated by $+$ or $-$ in the peg permutation representing the class. However, there are (up to symmetry) four other two-cell machines, those represented by the peg permutations $1^+2^-$, $2^+1^+$, $2^+1^-$, and $2^-1^+$. The last three can be shown to generate Wilf-equivalent classes by considering their corresponding generation sequences, and all can be shown to generate classes whose generating functions are algebraic.

It appears that the $\Av(123,231,312)$-machine of Section~\ref{sec-non-D-finite} is harder to model than the other four two-cell machines for the same reason the kernel method fails to apply: when there is a single entry in the $1^-$ cell and at least one entry in the $2^-$ cell, the act of popping the leftmost entry causes all entries in the $2^-$ cell to shift downward into the $1^-$ cell. While we now know that the Noonan--Zeilberger Conjecture is false thanks to the work of Garrabrant and Pak~\cite{Garrabrant:Pattern-avoidan:}, among all potential concrete counterexamples, the class $\Av(4123, 4231, 4312)$ analyzed in Section~\ref{subsec-Av-4123-4231-4312} is simplest yet identified.


\bigskip

\textbf{Acknowledgements:} We are grateful to Mireille Bousquet-M\'elou for suggesting a number of improvements to an earlier version of the paper, and to the referees for their careful reading of this work.

\bibliographystyle{acm}
\bibliography{../../refs}

\end{document}